\documentclass[10pt]{amsart}
\usepackage{amssymb,graphicx,mathabx}

\newtheorem{theorem}{Theorem}[section]
\newtheorem{lemma}[theorem]{Lemma}

\theoremstyle{definition}

\newtheorem{corollary}[theorem]{Corollary}
\newtheorem{proposition}[theorem]{Proposition}

\newtheorem{conjecture}[theorem]{Conjecture}

\theoremstyle{remark}
\newtheorem{remark}[theorem]{Remark}

\numberwithin{equation}{section}

\begin{document}

\title[Compactness of embedded free boundary minimal surfaces]{Compactness of the space of embedded minimal surfaces with free boundary in three-manifolds with nonnegative Ricci curvature and convex boundary}

\author[Ailana Fraser]{Ailana Fraser}
\address{Mathematics Department, University of British Columbia, 1984 Mathematics Road, Vancouver, BC V6T 1Z2, Canada}
\email{afraser@math.ubc.ca}

\author[Martin Li]{Martin Man-chun Li}
\address{Mathematics Department, University of British Columbia, 1984 Mathematics Road, Vancouver, BC V6T 1Z2, Canada}
\email{martinli@math.ubc.ca}

\thanks{This work was partially supported by NSERC}

\begin{abstract}
We prove a lower bound for the first Steklov eigenvalue of embedded minimal hypersurfaces with free boundary in a compact $n$-dimensional Riemannian manifold which has nonnegative Ricci curvature and strictly convex boundary. When $n=3$, this implies an apriori curvature estimate for these minimal surfaces in terms of the geometry of the ambient manifold and the topology of the minimal surface. An important consequence of the estimate is a smooth compactness theorem for embedded minimal surfaces with free boundary when the topological type of these minimal surfaces is fixed. 
\end{abstract}

\maketitle


\section{Introduction}

In a series of recent papers \cite{Fraser-Schoen12}, \cite{Fraser-Schoen13} and \cite{Fraser-Schoen11}, Fraser and Schoen studied an extremal problem for the first Steklov eigenvalue on compact surfaces with boundary and proved that minimal surfaces in Euclidean balls with free boundary on the boundary of the ball realize the extrema.  The equatorial disk \cite{Weinstock54} and the critical catenoid \cite{Fraser-Schoen12} \cite{Fraser-Schoen13} \cite{Fraser-Schoen11} uniquely maximize $\sigma_1(\Sigma)L(\partial \Sigma)$ among metrics on the disk and annulus respectively. In the recent preprint \cite{Fraser-Schoen12}, they were able to prove existence of extremal metrics for genus zero orientable surfaces with any number of boundary components, and the extrema are achieved by properly embedded minimal surfaces in the unit ball $B^3$ in $\mathbb{R}^3$ with free boundary. The non-orientable case of a M\"{o}bius band was also studied in detail. This motivates the question of finding more examples of properly embedded minimal surfaces in the unit ball. In the case without boundary, Lawson \cite{Lawson70a} proved that a closed orientable surface of any genus can be realized as an embedded minimal surface in the standard round sphere $S^3$ (while any non-orientable closed surface except $\mathbb{R}P^2$ can be realized as a minimal immersion into $S^3$). A central question in this direction is the following:

\vspace{.5cm}

\textbf{Question 1}: Which compact orientable surfaces with boundary can be realized as properly embedded minimal surfaces in the unit ball $B^3$ with free boundary? 

\vspace{.5cm}

Since $B^3$ is simply connected, any properly embedded surface in $B^3$ must be orientable. By the results of Fraser and Schoen \cite{Fraser-Schoen12} \cite{Fraser-Schoen13}, we know any genus zero orientable surface can be minimally embedded into $B^3$ as a free boundary solution. Therefore, it remains to look for surfaces of higher genus. Another related question one could ask is the following:

\vspace{.5cm}

\textbf{Question 2}: Given a compact orientable surface with boundary, in how many ways can it be realized as a properly embedded minimal surface in the unit ball $B^3$ with free boundary?

\vspace{.5cm}

Note that any such minimal surface comes in a continuous family because of the isometry group of the unit ball (which is a compact group). Therefore, we should look at minimal embeddings up to congruences. Using the holomorphicity of Hopf differential, Nitsche \cite{Nitsche85} proved that the equatorial totally geodesic disk is the only immersed minimal disk in $B^3$ with free boundary (up to congruences). For the annulus, we have the following conjecture:

\begin{conjecture}
The critical catenoid is the unique properly embedded minimal annulus in $B^3$ with free boundary up to congruences.
\end{conjecture}

One should compare this conjecture with the longstanding conjecture of Lawson \cite{Lawson70}, which asserts that the Clifford torus is the unique embedded minimal torus in $S^3$ up to congruences. Various partial results were obtained in this direction with additional assumptions (\cite{Montiel-Ros85} \cite{Ros95} \cite{Urbano90}). Very recently, Lawson's conjecture was proved in full generality by Brendle \cite{Brendle12} (see \cite{Brendle13} for a nice survey of this conjecture). 

In this paper, we prove that the space of properly embedded minimal surfaces in $B^3$ with free boundary is compact in the $C^\infty$ topology, if we fixed the topological type of the surface. In fact, we prove that the compactness result holds in any compact $3$-manifold $M^3$ with nonnegative Ricci curvature and strictly convex boundary $\partial M$. This result is similar to the classical compactness of minimal surfaces in closed manifolds with positive Ricci curvature of Choi and Schoen \cite{Choi-Schoen85}. Note that proper embeddedness is an essential assumption in our theorem.

\begin{theorem}
Let $M^3$ be a compact $3$-dimensional Riemannian manifold with nonempty boundary $\partial M$. Suppose $M$ has nonnegative Ricci curvature and the boundary $\partial M$ is strictly convex with respect to the inward unit normal. Then the space of compact properly embedded minimal surfaces of fixed topological type in $M$ with free boundary on $\partial M$ is compact in the $C^k$ topology for any $k \geq 2$.
\end{theorem}

The key ingredient in the proof is a lower bound on the first Steklov eigenvalue for properly embedded minimal surfaces with free boundary in terms of the boundary convexity of $\partial M$. Combining with an equality from \cite{Fraser-Schoen11}, this gives an apriori upper bound on the length of the boundary in terms of the topology of the minimal surface and the boundary convexity of $\partial M$. By an isoperimetric inequality of White \cite{White09}, we get an upper bound on the area of the minimal surface as well. These together give an apriori $L^2$ bound on the norm of the second fundamental form of the minimal surface. By a removable singularity theorem and curvature estimates similar to the ones in \cite{Choi-Schoen85}, we obtain the smooth compactness theorem above.

The outline of this paper is as follows. In section 2, we prove some general facts about minimal hypersurfaces with free boundary in a Riemannian $n$-manifold with nonnegative Ricci curvature and convex boundary. The key results are the isoperimetric inequality (Lemma 2.2) and the connectedness principle (Corollary 2.5). When $n=3$, we prove that any such manifold $M^3$ is diffeomorphic to the unit ball $B^3$ (Theorem 2.11). In section 3, we prove a lower bound for the first Steklov eigenvalue of a properly embedded minimal hypersurface in terms of the boundary convexity of the ambient manifold. Then, we specialize to dimension three and prove a removable singularity theorem and curvature estimates at the free boundary in sections 4 and 5. In section 6, we give a proof of our main compactness theorem (Theorem 1.2).

\textit{Acknowledgements}. The authors would like to thank Professor Richard Schoen for many useful discussions and his interest in this work. They would also like to express their gratitude to the anonymous referee for all the useful comments.


\section{Minimal hypersurfaces with free boundary}

Let $M^n$ be a compact $n$-dimensional Riemannian manifold with nonempty boundary $\partial M$. Let $\langle \cdot , \cdot \rangle$ be the metric on $M$ and $D$ be the Riemannian connection on $M$. The second fundamental form $h^{\partial M}$ of $\partial M$, with respect to the inner unit normal $n$, is given by  $h^{\partial M}(u,v)=\langle D_u v, n \rangle$, where $u,v$ are tangent to $\partial M$. The mean curvature $H^{\partial M}$ of $\partial M$ is then defined as the trace of $h^{\partial M}$; i.e., $H^{\partial M}=\sum_{i=1}^{n-1} h^{\partial M}(e_i,e_i)$ where $e_1,\ldots,e_{n-1}$ is any orthonormal basis for the tangent bundle $T\partial M$. All manifolds are assumed to be smooth up to the boundary unless otherwise stated.

Let $\varphi: \Sigma \to M$ be a compact hypersurface (possibly with boundary) properly immersed in $M$; that is, $\varphi(\partial \Sigma)=\varphi(\Sigma) \cap \partial M$. We say that $\Sigma$ is a minimal hypersurface with free boundary if $\Sigma$ is a minimal hypersurface (i.e., the mean curvature vanishes) and $\Sigma$ meets $\partial M$ orthogonally along $\partial \Sigma$. If $\varphi$ is an embedding, we treat $\Sigma \subset M$ as a submanifold of $M$ and take $\varphi$ to be the inclusion map $\Sigma \hookrightarrow M$. Suppose $\Sigma$ is two-sided; that is, there exists a globally defined unit normal vector field $N$ on $\Sigma$. Any normal vector field on $\Sigma$ has the form $X=fN$ for some $f \in C^\infty(\Sigma)$ and the second variation (see \cite{Schoen06} for example) of the volume functional with respect to $X=fN$ is 
\begin{align}
\delta^2 \Sigma (f)=&\int_\Sigma \left[ \|\nabla^\Sigma f\|^2- (\text{Ric}^M(N,N)+\|h^\Sigma\|^2)f^2 \right] d\text{Vol}_\Sigma \\
&- \int_{\partial \Sigma} h^{\partial M}(N,N) f^2 d \text{Vol}_{\partial \Sigma}, \notag
\end{align}
where $\nabla^\Sigma$ is the gradient operator on $\Sigma$, $\text{Ric}^M$ is the Ricci curvature of $M$, and $h^\Sigma$ is the second fundamental form of $\Sigma$ with respect to the unit normal $N$. Here, $d\text{Vol}_\Sigma$ and $d \text{Vol}_{\partial \Sigma}$ are the volume forms on $M$ and $\partial M$ respectively. Note that $N$ is tangent to $\partial M$ along $\partial \Sigma$ since $\Sigma$ meets $\partial M$ orthogonally along $\partial \Sigma$. We say that $\Sigma$ is stable if $\delta^2 \Sigma(f) \geq 0$ for any smooth function $f$ on $\Sigma$. Otherwise, $\Sigma$ is unstable. The following lemma is an immediate consequence of formula (2.1).

\begin{lemma}
Let $M^n$ be an $n$-dimensional compact Riemannian manifold with nonempty boundary $\partial M$. Suppose $M$ has nonnegative Ricci curvature and the boundary $\partial M$ is strictly convex with respect to the inward unit normal; i.e., there exists a constant $k>0$ such that $h^{\partial M}(u,u) \geq k>0$ for any unit vector $u$ tangent to $\partial M$. 

Then, any two-sided, properly immersed, smooth minimal hypersurface $\Sigma^{n-1}$ with nonempty free boundary $\partial \Sigma$ must be unstable. Moreover, if $M$ is orientable, then the $(n-1)$-th relative integral homology group $H_{n-1}(M,\partial M)$ vanishes.
\end{lemma}

\begin{proof}
Taking $f \equiv 1$ in (2.1), the curvature assumptions imply that $\delta^2 \Sigma(1) \leq -k \text{Vol}(\partial \Sigma) <0$ (since $\partial \Sigma \neq \emptyset$). Therefore, $\Sigma$ is unstable. To prove the second assertion, suppose $H_{n-1}(M,\partial M) \neq 0$. Let $\alpha \neq 0$ be a nontrivial primitive class in $H_{n-1}(M,\partial M)$. We claim that one can choose a compact embedded minimal hypersurface $\Sigma$ with free boundary (which may be empty) such that $\alpha=[\Sigma]$ in $H_{n-1}(M,\partial M)$ and $\Sigma$ minimizes volume among all hypersurfaces homologous to $\Sigma$ relative to $\partial M$ (see, for example, Corollary 9.9 in \cite{Federer-Fleming60}). This can be seen as follows. By Poincare-Lefschetz duality and noting $K(\mathbb{Z},1)=S^1$, we have 
\begin{equation*}
H_{n-1}(M,\partial M) \cong H^1(M) \cong \langle M, S^1 \rangle.
\end{equation*}
So $\alpha$ corresponds to a map $f:M \to S^1$, which we can assume to be smooth by the Whitney approximation theorem. Then, for any regular value $z \in S^1$ of $f$, the preimage $f^{-1}(z)$ is a properly embedded compact orientable hypersurface which represents $\alpha$ in $H_{n-1}(M,\partial M)$. Since $\Sigma$ is stable, the curvature estimates for stable minimal hypersurfaces ( \cite{Gruter-Jost86a} \cite{Schoen-Simon-Yau75} \cite{Schoen-Simon81}) imply that $\Sigma$ is smooth up to the boundary (which may be empty), except possibly along a singular set $\mathcal{S}$ of Hausdorff dimension at most $n-8$. Assume first the singular set $\mathcal{S}$ is empty. If $\partial \Sigma \neq \emptyset$, then we have a contradiction with the above statement since $\Sigma$ is stable and two-sided. If $\partial \Sigma = \emptyset$, then it follows from Lemma 2.2 below that this is impossible. In case $\mathcal{S}$ is nonempty, a cutoff argument near the singular set gives the same conclusion since $S$ has codimension greater than $3$.
\end{proof}

The next lemma is an isoperimetric inequality for minimal hypersurfaces in $M$, which holds under slightly weaker curvature assumptions than those in Lemma 2.1.

\begin{lemma}[Isoperimetric inequality]
Let $M^n$ be a compact $n$-dimensional Riemannian manifold with nonempty boundary $\partial M$. Suppose $M$ has nonnegative Ricci curvature and the boundary $\partial M$ is strictly mean convex with respect to the inward unit normal.

Then, $M$ contains no smooth, closed, embedded minimal hypersurface. Furthermore, if $n \leq 7$, then there exists a constant $c>0$, depending only on $M$, such that 
\begin{equation}
\text{Vol }(\Sigma) \leq c \text{Vol }(\partial \Sigma)
\end{equation}
for any smooth immersed minimal hypersurface $\Sigma$ in $M$. 
\end{lemma}

\begin{proof}
It suffices to show that if $M$ contains no smooth, closed, embedded minimal hypersurface, then the isoperimetric inequality (2.2) follows from Theorem 2.1 of \cite{White09}. Suppose not, and let $\Sigma$ be a smooth, closed embedded minimal hypersurface in $M$. Since $\partial M$ is strictly mean convex, we have $\Sigma \cap \partial M = \emptyset$ by the strong maximum principle (see \cite{White09}). Therefore, $\Sigma$ and $\partial M$ are a positive distance apart; i.e., $d(\Sigma,\partial M)=\ell>0$. Since both $\partial M$ and $\Sigma$ are compact, there exists a minimizing geodesic $\gamma:[0,\ell] \to M$ (parametrized by arc length) that realizes the distance between $\Sigma$ and $\partial M$. Since $\gamma$ is minimizing, $\gamma$ lies in the interior of $M$ except at $\gamma(\ell) \in \partial M$. Moreover, $\gamma$ is orthogonal to $\Sigma$ and $\partial M$ at the end points. Pick any orthonormal basis $e_1,\ldots,e_{n-1}$ for $T_{\gamma(0)} \Sigma$ and let $V_1,\ldots,V_{n-1}$ be their parallel extensions to normal vector fields  along $\gamma$. If we look at the second variation of $\gamma$ with respect to the normal variation fields $V_i$, $i=1,\ldots, n-1$ and sum over $i$, we get
\begin{equation*}
\sum_{i=1}^{n-1} \delta^2 \gamma(V_i,V_i)=-\int_0^\ell \text{Ric}^M(\gamma',\gamma') ds - H^\Sigma(\gamma(0))-H^{\partial M}(\gamma(\ell)) < 0
\end{equation*}
since $\text{Ric}^M \geq 0$, $\Sigma$ is minimal, and $\partial M$ is strictly mean convex. Therefore, $\delta^2 \gamma(V_i,V_i) <0$ for some $i$ and hence $\gamma$ cannot be stable. This contradicts that $\gamma$ is a minimizing geodesic from $\Sigma$ to $\partial M$. This proves the lemma.
\end{proof}

\begin{remark}
We will later apply (2.2) to properly embedded minimal surfaces with free boundary. However, the isoperimetric inequality (2.2) applies in general to any minimal hypersurface $\Sigma$ without any assumptions on the boundary $\partial \Sigma$.
\end{remark}

The next lemma shows that under the curvature assumptions in Lemma 2.1, any two properly embedded minimal hypersurfaces must intersect each other.

\begin{lemma}
Let $M^n$ be an $n$-dimensional compact orientable Riemannian manifold with nonempty boundary $\partial M$. Suppose $M$ has nonnegative Ricci curvature and the boundary $\partial M$ is strictly convex with respect to the inward unit normal. Then, any two properly embedded orientable minimal hypersurfaces $\Sigma_1$ and $\Sigma_2$ in $M$ with free boundaries on $\partial M$ must intersect; i.e., $\Sigma_1 \cap \Sigma_2 \neq \emptyset$. 
\end{lemma}

\begin{proof}
We argue by contradiction. Suppose there are two disjoint properly embedded minimal hypersurfaces $\Sigma_1$, $\Sigma_2$ with free boundaries on $\partial M$. Without loss of generality, we assume that both $\Sigma_1$ and $\Sigma_2$ are connected. Note that $\partial \Sigma_1$ and $\partial \Sigma_2$ are nonempty by Lemma 2.2. Since $H_{n-1}(M,\partial M)=0$ by Lemma 2.1, there exists a compact connected domain $\Omega \subset M$ such that $\partial \Omega=\Sigma_1 \cup \Sigma_2 \cup \Gamma$, where $\Gamma$ is a smooth domain in $\partial M$. On $\Omega$, let $d_1$ and $d_2$ be the distance functions from $\Sigma_1$ and $\Sigma_2$ respectively. Since $M$ has nonnegative Ricci curvature and $\Sigma_1$, $\Sigma_2$ are minimal, we have $\Delta^M d_1 \leq 0$ and $\Delta^M d_2 \leq 0$ on $\Omega$ in the barrier sense (see Definition 1 in \cite{Calabi58}) away from $\Sigma_1$ and $\Sigma_2$. Notice that we have used the fact that these minimal hypersurfaces meet $\partial M$ orthogonally so that for any point $x$ in $\Omega \setminus \Sigma_1$, $d_1(x)$ is realized by a geodesic from $x$ to an interior point $y$ on $\Sigma_1$. Hence, the same calculation as in the Laplace comparison theorem for the case without boundary applies here. Therefore, $\Delta^M (d_1+d_2) \leq 0$ in the barrier sense on $\Omega$. We claim that $d_1+d_2$ is constant on $\Omega$. 

If $d_1+d_2$ attains an interior minimum in $\Omega$, the generalized Hopf maximum principle in \cite{Calabi58} implies that $d_1+d_2$ is constant. Since $d_1+d_2$ is continuous, the global minimum must be achieved by some point $p \in \partial \Omega=\Sigma_1 \cup \Sigma_2 \cup \Gamma$. Since $\partial M$ is strictly convex, the outward normal derivative of $d_1+d_2$ is strictly positive on $\Gamma \subset \partial M$. Indeed the outward normal derivative of each of $d_1$ and $d_2$ is positive on $\Gamma$.  Therefore, the minimum $p$ lies on $\Sigma_1$ or  $\Sigma_2$. Assume, without loss of generality, that $p \in \Sigma_1$. Observe that $d_1+d_2=d_2$ on $\Sigma_1$, and $p$ is a point on $\Sigma_1$ that is closest to $\Sigma_2$. Since $\partial M$ is convex and $\Omega$ is connected, there exists a minimizing geodesic $\gamma \subset \Omega$ that connects $p$ to $\Sigma_2$, and such a geodesic $\gamma$ is disjoint from $\partial M$ and meets $\Sigma_1$ and $\Sigma_2$ orthogonally at the end points. On the other hand, as $d_1+d_2=d_1$ on $\Sigma_2$, $\gamma$ actually realizes the distance between $\Sigma_1$ and $\Sigma_2$ in $\Omega$. This implies that $d_1+d_2$ is constant on $\gamma$, but then $d_1+d_2$ has an interior minimum, which implies that $d_1+d_2$ is constant in $\Omega$. However, since the outward normal derivative of $d_1+d_2$ is strictly positive on $\Gamma \subset \partial M$ and $\Gamma$ is nonempty, this is a contradiction. Therefore, $\Sigma_1$ and $\Sigma_2$ cannot be disjoint. 
\end{proof}

\begin{corollary}[Connectedness principle]
Under the curvature assumptions on $M$ and $\partial M$ in Lemma 2.4, any properly embedded minimal hypersurface in $M$ with free boundary is connected.
\end{corollary}

One can prove Lemma 2.4 using a form of Reilly's formula \cite{Reilly77}. In this paper, we will look at compact manifolds with piecewise smooth boundary. We first observe that Reilly's formula holds for such manifolds provided that the function is smooth enough away from the singular set $S$ of the boundary. Note that there is a sign difference in the formula from that in \cite{Choi-Wang83}, since we are using the inward unit normal instead of the outward unit normal.

\begin{lemma}[Reilly's formula]
Let $\Omega$ be a compact $n$-manifold with piecewise smooth boundary $\partial \Omega=\cup_{i=1}^k \Sigma_i$. Suppose $f$ is a continuous function on $\Omega$ where $f \in C^\infty(\Omega \setminus S)$, and $S=\cup_{i=1}^k \partial \Sigma_i$ is the singular set. Assume that there exists some $C>0$, depending only on $f$, such that $\|f\|_{C^3(\Omega')} \leq C$ for all $\Omega' \subset \subset \Omega \setminus S$. Then, Reilly's formula holds:
\begin{align}
0 &= \int_\Omega \text{Ric}^\Omega (D f,D f) - (\Delta^\Omega f)^2 +\|D^2 f\|^2 \\
&+ \sum_{i=1}^k \int_{\Sigma_i} (-\Delta^{\Sigma_i} f + H^{\Sigma_i} \frac{\partial f}{\partial n_i}) \frac{\partial f}{\partial n_i} + \langle \nabla^{\Sigma_i} f,\nabla^{\Sigma_i} \frac{\partial f}{\partial n_i} \rangle+h^{\Sigma_i}(\nabla^{\Sigma_i} f,\nabla^{\Sigma_i} f). \notag
\end{align}
Here, $\text{Ric}^\Omega$ is the Ricci tensor of $\Omega$; $\Delta^\Omega, D^2$ and $D$ are the Laplacian, Hessian, and gradient operators on $\Omega$ respectively; $\Delta^{\Sigma_i}$ and $\nabla^{\Sigma_i}$ are the intrinsic Laplacian and gradient operators on each $\Sigma_i$; $n_i$ is the inward unit normal of $\Sigma_i$; $H^{\Sigma_i}$ and $h^{\Sigma_i}$ are the mean curvature and second fundamental form of $\Sigma_i$ in $\Omega$ with respect to the inward unit normal respectively. 
\end{lemma}

\begin{proof}
Since the singular set $S=\cup_{i=1}^k \partial \Sigma_i$ has codimension two in $\Omega$, the smoothness assumption of $f$ implies that Stokes' Theorem with singularites (Theorem 3.3 of \cite{Lang02}) is applicable. Hence, the same proof as in Theorem 1 of \cite{Choi-Wang83} gives the desired result. 
\end{proof}

Using Lemma 2.6, we give an alternative proof of Lemma 2.4.

\begin{proof}[Alternative Proof of Lemma 2.4]
We will prove Lemma 2.4 by contradiction. Suppose $\Sigma_1$ and $\Sigma_2$ are connected and disjoint. Let $\Omega$ be the connected domain bounded by $\Sigma_1$ and $\Sigma_2$ modulo $\partial M$ as before. Note that $\Omega$ is a compact $n$-manifold with piecewise smooth boundary $\partial \Omega=\Sigma_1 \cup \Sigma_2 \cup \Gamma$, where $\Gamma \subset \partial M$. Let int$(\Omega)$ denote the interior of $\Omega$. Consider the following mixed Dirichlet-Neumann boundary value problem on $\Omega$:
\begin{equation}
\begin{cases}
\Delta^\Omega f=0 & \text{ on int}(\Omega)\\
f=0 & \text{ on } \Sigma_1, \\
f=1 & \text{ on } \Sigma_2, \\
\frac{\partial f}{\partial n}=0 & \text{ on } \Gamma,
\end{cases}
\end{equation}
where $n$ is the outward unit normal on $\partial M$. Since $\Sigma_1$ and $\Sigma_2$ meet $\partial M$ orthogonally along their boundaries, there exists a function $\varphi \in C^\infty(\Omega)$ such that 
\begin{equation*}
\begin{cases}
\varphi=0 & \text{ on } \Sigma_1,\\
\varphi=1 & \text{ on } \Sigma_2,\\
\frac{\partial \varphi}{\partial n}=0 & \text{ on } \Gamma.
\end{cases}
\end{equation*}
Letting $\hat{f}=f-\varphi$, the mixed boundary value problem (2.4) is equivalent to the following mixed boundary value problem with zero boundary data:
\begin{equation}
\begin{cases}
\Delta^\Omega \hat{f}= \Delta^\Omega \varphi & \text{ on int} (\Omega), \\
\hat{f}=0 & \text{ on } \Sigma_1 \cup \Sigma_2,\\
\frac{\partial \hat{f}}{\partial n}=0 & \text{ on } \Gamma.
\end{cases}
\end{equation}
Since $\Delta^\Omega \varphi \in C^\infty(\Omega)$, classical results for elliptic equations with homogeneous boundary data (\cite{Agmon-Douglis-Nirenberg59} and \cite{Lieberman86}) imply that a solution to (2.5) exists in the classical sense and the solution $\hat{f} \in C^{0,\alpha}(\Omega) \cap C^\infty(\Omega \setminus S)$, where $S=\partial \Sigma_1 \cup \partial \Sigma_2$, and therefore $f=\hat{f}+\varphi \in C^{0,\alpha}(\Omega) \cap C^\infty(\Omega \setminus S)$ is a solution to (2.4), with uniform $C^3$ estimates away from the singular set $S$. Applying Reilly's formula in Lemma 2.6 to $f$ and $\Omega$, as $\Delta^\Omega f=0$, we obtain
\begin{equation}
0 \geq \int_\Omega \text{Ric}^M (D f, D f) +\int_\Gamma h^{\partial M}(\nabla^{\partial M} f, \nabla^{\partial M}f).
\end{equation}
The boundary terms for $\Sigma_1$ and $\Sigma_2$ vanish since $\Sigma_1$ and $\Sigma_2$ are minimal and $f$ is constant on $\Sigma_1$ and $\Sigma_2$. Since $\text{Ric}^M \geq 0$ and $h^{\partial M} \geq k>0$, (2.6) implies that $\nabla^{\partial M}f=0$ and hence $f$ is locally constant on $\Gamma$, which is impossible since $\partial \Gamma = \partial \Sigma_1 \cup \partial \Sigma_1$ and $f=0$ on $\Sigma_1$ but $f=1$ on $\Sigma_2$. Thus, we have a contradiction. 
\end{proof}
 
\begin{remark}
\begin{itemize}
\item[(a)] Note that the free boundary condition comes in the proof in a rather subtle way that gives enough regularity to the mixed boundary value problem (2.4) to apply Reilly's formula in Lemma 2.6. If the free boundary condition is dropped, the theorem is no longer true. For example, there are disjoint flat disks in the unit ball in $\mathbb{R}^3$. 
\item[(b)] The theorem does not hold if we only assume that $\text{Ric}^M \geq 0$ and $\partial M$ is only weakly convex. For example, let $B^{n-1}$ be the $(n-1)$-dimensional unit ball in $\mathbb{R}^n$ and $M$ be the product manifold $B^{n-1} \times [0,1]$ smoothly capped off by two unit half $n$-balls at the two ends. Then all the slices $B^{n-1} \times \{t\}$, $t \in [0,1]$, are mutually disjoint embedded minimal hypersurfaces with free boundary on $\partial M$. However, there are some rigidity results in this case (see \cite{Petersen-Wilhelm03}).
\end{itemize}
\end{remark}

We now prove a result about the connectedness of $\partial M$ for $M$ with nonnegative curvature and strictly mean convex boundary. 

\begin{proposition}
Let $M^n$ be an $n$-dimensional compact Riemannian manifold with nonempty boundary $\partial M$. Suppose $M$ has nonnegative Ricci curvature and the boundary $\partial M$ is strictly mean convex with respect to the inward unit normal. Then $\partial M$ is connected, and the homomorphism 
\begin{equation*}
\pi_1(\partial M) \xrightarrow{i_*} \pi_1(M)
\end{equation*}
induced by the inclusion map $i:\partial M \to M$ is surjective.
\end{proposition} 

\begin{proof}
The proof is similar to the one in Lemma 2.2. If $\partial M$ is not connected, there exists a minimizing geodesic from one component of $\partial M$ to another component that realizes the distance between them. However, the second variation formula and the curvature assumptions on $M$ and $\partial M$ imply that $\gamma$ is unstable, which is a contradiction. Therefore, $\partial M$ is connected. (One can also give a different proof using Reilly's formula (2.3). Suppose $\partial M$ is not connected. Let $\Sigma$ be one of its components. Take $f \in C^\infty(M)$ to be a harmonic function that is equal to one on $\Sigma$ and is equal to zero on $\partial M \setminus \Sigma$, which is nonempty. Then, Reilly's formula implies that $f$ is constant, which is a contradiction.) The same argument applies to the universal cover $\tilde{M}$ of $M$. Thus $\partial \tilde{M}$ is connected and this implies the surjectivity of the homomorphism $\pi_1(\partial M) \to \pi_1(M)$ as in \cite{Lawson70}.
\end{proof}

\begin{remark}
As noted in \cite{Lawson70}, Proposition 2.8 remains true if $\partial M$ is assumed to be only piecewise smooth and if the interior angle between two smooth boundary pieces is always less than $\pi$.
\end{remark}

From the alternative proof of Lemma 2.4, we actually proved that if $\Omega$ is a connected $n$-manifold with nonnegative Ricci curvature and piecewise smooth boundary $\partial \Omega$, where we can decompose $\partial \Omega=\Gamma_1 \cup \Gamma_2$ with $\Gamma_1$ non-empty and strictly convex with respect to the inward unit normal, and $\Gamma_2$ minimal, then $\Gamma_2$ must be connected. The following corollary is an immediate consequence following the arguments in the proof of Theorem 2 in \cite{Lawson70}.

\begin{corollary}
Let $M^n$ be a compact $n$-dimensional Riemannian manifold with nonempty boundary $\partial M$. Suppose $M$ has nonnegative Ricci curvature and the boundary $\partial M$ is strictly convex with respect to the inward unit normal. Let $\Sigma$ be a properly embedded minimal hypersurface in $M$ with free boundary on $\partial M$. If both $\Sigma$ and $M$ are orientable, then $\Sigma$ divides $M$ into two connected components $\Omega_1$ and $\Omega_2$.
\end{corollary}

When $n=3$, we get much stronger topological restrictions on $M^3$ from the curvature and boundary convexity assumptions.

\begin{theorem}
Let $M^3$ be a compact Riemannian 3-manifold with nonempty boundary $\partial M$. Assume $M$ has nonnegative Ricci curvature. 
\begin{itemize}
\item[(a)] If $M$ is orientable and $\partial M$ is strictly mean convex with respect to the inner unit normal, then $M^3$ is diffeomorphic to a $3$-dimensional handlebody.
\item[(b)] If $\partial M$ is strictly convex with respect to the inner unit normal, then $M^3$ is diffeormorphic to the $3$-ball $B^3$.
\end{itemize}
\end{theorem}

\begin{remark}
Note that we do not need to assume $M$ is orientable in case (b); it follows as a consequence. For higher dimensions $n \geq 4$, we conjecture that in case (b), $M^n$ has finite fundamental group (see Conjecture 1.3 in \cite{Li12}).
\end{remark}

\begin{proof}
First, we assume that $M$ is orientable. If $\partial M$ is strictly mean convex and nonempty, then it is connected, by Proposition 2.8. Using Theorem 5 in \cite{Meeks-Simon-Yau82}, we know that $M$ is a handlebody. If $\partial M$ is strictly convex, then we also have  $H_2(M,\partial M)=0$ from Lemma 2.1, which implies that $M$ is diffeomorphic to the $3$-ball $B^3$. 

Suppose $M$ is non-orientable and $\partial M$ is strictly convex. Then the orientable double cover $\tilde{M}$ is the $3$-ball $B^3$. Therefore, $\partial \tilde{M}=S^2$ is a double cover of $\partial M$; thus $\partial M$ is homeomorphic to $\mathbb{R}P^2$. However, since $\partial M$ is the boundary of a compact manifold, by a theorem of Pontrjagin \cite{Pontryagin50}, all the Stiefel-Whitney numbers of $\partial M$ vanish. However $w_1(\mathbb{R}P^2)=w_2(\mathbb{R}P^2)=1$. This is a contradiction. So $M$ must be orientable.
\end{proof}


\section{Steklov Eigenvalue Estimate}

In this section, we prove a lower bound for the first Steklov eigenvalue of a compact properly embedded minimal hypersurface satisfying the free boundary condition in a compact orientable manifold $M$ with boundary, where $M$ has nonnegative Ricci curvature and $\partial M$ is strictly convex. We refer the reader to section 2 of \cite{Fraser-Schoen11} for a brief description of the Dirichlet-to-Neumann map and Steklov eigenvalues.

\begin{theorem}
Let $M^n$ be an $n$-dimensional compact orientable Riemannian manifold with nonempty boundary $\partial M$. Suppose $M$ has nonnegative Ricci curvature and the boundary $\partial M$ is strictly convex with respect to the inward unit normal. Let $k>0$ be a constant such that $h^{\partial M}(u,u) \geq k >0$ for any unit vector $u$ tangent to $\partial M$. 

Let $\Sigma$ be a properly embedded minimal hypersurface in $M$ with free boundary on $\partial M$. Supose one of the following holds,
\begin{itemize}
\item[(i)] $\Sigma$ is orientable; or
\item[(ii)] $\pi_1(M)$ is finite;
\end{itemize}
then we have the following eigenvalue estimate
\begin{equation*}
\sigma_1(\Sigma) \geq \frac{k}{2},
\end{equation*}
where $\sigma_1(\Sigma)$ is the first non-zero Steklov eigenvalue of the Dirichlet-to-Neumann map on $\Sigma$. 
\end{theorem}

\begin{proof}
We first assume that $\Sigma$ is orientable. By Corollary 2.5 and Corollary 2.10, $\Sigma$ is connected and $\Sigma$ divides $M$ into two connected components $\Omega_1$ and $\Omega_2$. Take $\Omega=\Omega_1$. Let $\partial \Omega = \Sigma \cup \Gamma$ where $\Gamma \subset \partial M$. Thus, $\partial \Sigma=\partial \Gamma$. Note that $\Gamma$ is not necessarily connected, but each component of $\Gamma$ must intersect $\Sigma$ along some component of $\partial \Sigma$. Otherwise, $\partial M$ would have more than one component, which would contradict Proposition 2.8.

Let $z \in C^\infty( \partial \Sigma)$ be a first eigenfunction of the Dirichlet-to-Neumann map on $\Sigma$; i.e., there exists $z_1 \in C^\infty(\Sigma)$ such that 
\begin{equation}
\begin{cases}
\Delta^\Sigma z_1 =0  & \text{ on } \Sigma, \\
z_1=z & \text{ along } \partial \Sigma, \\
\frac{\partial z_1}{\partial \nu_\Sigma} = \sigma_1 z  &  \text{ along } \partial \Sigma,
\end{cases}
\end{equation}
where $\nu_\Sigma$ is the outward conormal vector of $\partial \Sigma$ with respect to $\Sigma$, and $\sigma_1=\sigma_1(\Sigma)$. Recall that $\partial \Gamma=\partial \Sigma$. Let $z_2 \in C^\infty(\Gamma)$ be the harmonic extension of $z \in C^\infty(\partial \Gamma)$ to $\Gamma$:
\begin{equation*}
\begin{cases}
\Delta^{\Gamma} z_2=0 & \text{ on } \Gamma, \\
z_2=z & \text{ along } \partial \Gamma.
\end{cases}
\end{equation*}
Next, we consider the Dirichlet boundary value problem on the compact $n$-manifold $\Omega$ with piecewise smooth boundary $\partial \Omega=\Sigma \cup \Gamma$:
\begin{equation}
\begin{cases}
\Delta^\Omega f = 0 & \text{ on } \Omega, \\
f=z_1 & \text{ along } \Sigma, \\
f=z_2 & \text{ along } \Gamma .
\end{cases}
\end{equation}
Note that the Dirichlet boundary data is continuous. Standard results on elliptic boundary problems (\cite{Agmon-Douglis-Nirenberg59}, \cite{Azzam80}, \cite{Azzam81}) imply that a classical solution for (3.2) exists and $f \in C^{1,\alpha}(\Omega) \cap C^\infty(\Omega \setminus \partial \Sigma)$, for every $\alpha \in (0,1)$, together with uniform $C^3$ estimates away from the singular set $\partial \Sigma$. Applying Reilly's formula (2.3), we have
\begin{equation*}
0 \geq \int_{\Sigma} \left( \langle \nabla^\Sigma f, \nabla^\Sigma \frac{\partial f}{\partial n_\Sigma} \rangle +h^{\Sigma}(\nabla^\Sigma f,\nabla^\Sigma f) \right) + \int_{\Gamma} \left( \langle \nabla^\Gamma f, \nabla^\Gamma \frac{\partial f}{\partial n_\Gamma} \rangle + k \|\nabla^\Gamma f\|^2 \right).
\end{equation*}
where $n_\Sigma$ and $n_\Gamma$ are the inward unit normals of $\Sigma$ and $\Gamma$ respectively, with respect to $\Omega$. Without loss of generality, we can assume that the integral $\int_{\Sigma} h^{\Sigma}(\nabla^\Sigma f,\nabla^\Sigma f) \geq 0$. Otherwise, we choose $\Omega=\Omega_2$ instead. Using the fact that $\Delta^\Sigma(f|_\Sigma)=\Delta^\Sigma z_1=0$ and $\Delta^\Gamma (f|_\Gamma)=\Delta^\Gamma z_2=0$, integrating by parts gives
\begin{equation}
0 \geq \int_{\partial \Sigma} \frac{\partial f}{\partial \nu_\Sigma} \frac{\partial f}{\partial n_\Sigma} + \int_{\partial \Gamma} \frac{\partial f}{\partial \nu_\Gamma} \frac{\partial f}{\partial n_\Gamma} + k \int_{\Gamma} \|\nabla^\Gamma f\|^2,
\end{equation}
where $\nu_\Sigma$ and $\nu_\Gamma$ are the outward conormal vectors of $\partial \Sigma=\partial \Gamma$ with respect to $\Sigma$ and $\Gamma$ respectively. Since $\Sigma$ meets $\Gamma$ orthogonally along $\partial \Sigma=\partial \Gamma$, we have $\nu_\Sigma=-n_\Gamma$ and $n_\Sigma=-\nu_\Gamma$ along the common boundary $\partial \Sigma$. Since $f \in C^{1,\alpha}(\Omega)$, the gradient $D f$ is continuous on $\Omega$ up to the singular set $\partial \Sigma=\partial \Gamma$. Therefore,
\begin{equation*}
\int_{\partial \Sigma} \frac{\partial f}{\partial \nu_\Sigma} \frac{\partial f}{\partial n_\Sigma}=-\int_{\partial \Sigma} \frac{\partial f}{\partial \nu_\Sigma} \frac{\partial f}{\partial \nu_\Gamma}=\int_{\partial \Gamma} \frac{\partial f}{\partial \nu_\Gamma} \frac{\partial f}{\partial n_\Gamma}.
\end{equation*}
Putting this back into (3.3) and using the boundary condition in (3.1), we get 
\begin{equation*}
0 \geq -2 \sigma_1 \int_{\partial \Gamma} f \frac{\partial f}{\partial \nu_\Gamma} + k \int_{\Gamma} \|\nabla^\Gamma f\|^2.
\end{equation*}
Since $\Delta^\Gamma (f|_\Gamma)=0$, another integration by parts on $\Gamma$ implies that $\int_{\partial \Gamma} f \frac{\partial f}{\partial \nu_\Gamma} =\int_{\Gamma} \|\nabla^\Gamma f\|^2$. As $f$ is non-constant on $\Gamma$ (since $z$ is non-constant on $\partial \Gamma$), we get $\sigma_1 \geq k/2$. This proves the theorem when $\Sigma$ is orientable.

Now, suppose $\Sigma$ is not orientable but $\pi_1(M)$ is finite. Let $\tilde{M}$ be the universal cover of $M$. Then $\tilde{M}$ satisfies the same curvature assumptions as $M$. Since $\pi_1(M)$ is finite, $\tilde{M}$ is compact and $\pi: \tilde{M} \to M$ is a finite covering. Let $\tilde{\Sigma}$ be the lifting of $\Sigma$; i.e., $\tilde{\Sigma}=\pi^{-1}(\Sigma)$. Since $\tilde{M}$ is simply connected and $\tilde{\Sigma}$ is properly embedded, both $\tilde{M}$ and $\tilde{\Sigma}$ are orientable. By the result above, $\sigma_1(\tilde{\Sigma}) \geq k/2$. But the pullback by $\pi$ of the first Steklov eigenfunction of $\Sigma$ into $\tilde{\Sigma}$ is again an eigenfunction of $\tilde{\Sigma}$. Therefore, $\sigma_1(\Sigma) \geq \sigma_1(\tilde{\Sigma}) \geq k/2$. The proof of Theorem 3.1 is completed.
\end{proof}

Since $B^n$ is simply connected, we have the following corollary. 

\begin{corollary}
Let $\Sigma$ be a compact properly embedded minimal hypersurface in $B^n$, the Euclidean unit ball, with free boundary on $\partial B^n$. Then $\sigma_1(\Sigma) \geq 1/2$.  
\end{corollary}

It is known (\cite{Fraser-Schoen11}) that for a minimal submanifold properly immersed in the unit ball in $\mathbb{R}^n$ with free boundary on the unit sphere, the coordinate functions are Steklov eigenfunctions with eigenvalue $1$. It is natural to ask if this is the first Steklov eigenvalue when the minimal submanifold is properly embedded and has codimension one.

\begin{conjecture}
Let $\Sigma$ be a compact properly embedded minimal hypersurface in $B^n$, the Euclidean unit ball, with free boundary on $\partial B^n$. Then $\sigma_1(\Sigma) =1$.  
\end{conjecture}

From now on, we will assume that $n=3$. In \cite{Fraser-Schoen11}, Fraser and Schoen proved that if $\Sigma$ is a compact orientable surface of genus $g$ with $\gamma$ boundary components of total length $L(\partial \Sigma)$, then $\sigma_1(\Sigma) L(\partial \Sigma) \leq 2\pi (g+\gamma)$. Combining this with a bound of Kokarev \cite{Kokarev11} and Theorem 3.1, we get the following estimate on the boundary length of a minimal surface with free boundary in terms of its topology.

\begin{proposition}
Let $M$ and $\Sigma$ be the same as in Theorem 3.1. Assume that dim $\Sigma=2$. Then, 
\begin{equation*}
L(\partial \Sigma) \leq \min \left\{\frac{4\pi}{k}(g+\gamma), \frac{16\pi}{k} \left[ \frac{g+3}{2} \right] \right\}.
\end{equation*}
\end{proposition}

\begin{remark}
By Theorem 2.11, we know that $M^3$ is diffeomorphic to the unit ball $B^3$, which is simply connected, and hence any properly embedded surface in $M$ is automatically orientable.
\end{remark}

\begin{corollary}
Let $\Sigma$ be a compact properly embedded minimal hypersurface in $B^3$, the Euclidean unit ball, with free boundary on $\partial B^3$. Then $L(\partial \Sigma) \leq 4\pi(g+\gamma)$.
\end{corollary}


\section{Removable Singularity Theorem}

In this section, we prove a removable singularity result at the free boundary for properly embedded minimal surfaces with free boundary in a compact Riemannian 3-manifold with boundary. 

\begin{theorem}
Let $M^3$ be a Riemannian 3-manifold with boundary and let $Q$ be a point on $\partial M$. Suppose $\Sigma \subset M$ is a (possibly non-orientable) minimal surface with smooth boundary and finite Euler characteristic which is properly embedded in $M \setminus \{ Q\}$. Assume that $\Sigma$ meets $\partial M$ orthogonally along $\partial \Sigma$. If $Q$ lies in the closure of $\partial \Sigma$ as a point set, then $\Sigma \cup \{ Q\}$ is a smooth properly embedded minimal surface in $M$.
\end{theorem}

\begin{proof}
Assume first that $\Sigma$ is orientable. Let $\Omega$ denote the Riemann surface (with boundary) determined by the induced metric on $\Sigma$, and let $F:\Omega \to M$ be a conformal harmonic embedding with $\Sigma=F(\Omega)$. Since the Euler characteristic of $\Sigma$ is finite, $\Omega$ is conformally equivalent to a compact Riemann surface (with boundary) with a finite number of disks and points removed. Therefore, there exist (open or closed) arcs or points $\gamma_1,\ldots,\gamma_k$ such that $\overline{\Omega}=\Omega \cup (\cup_{i=1}^k \gamma_i )$ is a compact Riemann surface $\overline{\Omega}$ with boundary. Since $Q$ lies in the closure of $\partial \Sigma$ as a point set, this implies that we can extend $F$ continuously to $\overline{\Omega}$. We claim that all $\gamma_i$ are points on the boundary. Note that for each $\gamma_i$, $F(\gamma_i)=\{Q\}$. Suppose $\gamma_i$ is an arc. Since $F$ is continuous up to $\gamma_i$ and is harmonic with a constant value along $\gamma_i$, by a result in \cite{Gruter97}, $F$ is $C^{1,\alpha}$ up to $\gamma_i$. Since $F$ is conformal, $dF=0$ along $\gamma_i$ and therefore we can extend $F$ past $\gamma_i$ to take the constant value $Q$ and this extension is still $C^1$. This results in a weakly harmonic map which is $C^1$, and therefore a classical harmonic map (\cite{Hildebrandt-Kaul-Widman77}). However, since $F$ is constant on an open set, $F$ must be identically constant on $\Omega$, which is a contradiction. Therefore, each $\gamma_i$ is a point. Moreover, we see that $\gamma_i$ is not an interior point, since $F$ is a proper  embedding and $Q$ lies in the closure of $\partial \Sigma$. Therefore, $F$ extends smoothly (\cite{Jost86b}) across $\gamma_i$ to a harmonic map from $\overline{\Omega}$ to $M$. If $\gamma_i$ were a boundary branch point of $F$, then by the asymptotic expansion near a branch point at the free boundary (Lemma 1 of \cite{Gulliver-Lesley73}), there would be a line of self-intersection emanating from $Q$, which contradicts that $\Sigma$ is embedded. Therefore, $F$ extends as a proper minimal immersion from $\Omega \cup \{ \gamma_1,\ldots,\gamma_k\}$. Since $\Sigma$ is properly embedded, the maximum principle for minimal surfaces with free boundary implies that $k=1$ (otherwise, there would be two minimal half-disks with free boundary that touch at one point at $Q$, which would violate the maximum principe). Hence, $F:\Omega \cup \{\gamma_1\} \to \Sigma \cup \{Q\}$ is a smooth properly embedded compact minimal surface in $M$ with smooth free boundary on $\partial M$.

Now suppose $\Sigma$ is not orientable. Let $\tilde{\Sigma}$ be the orientable double cover of $\Sigma$ and let $\Omega$ be the Riemann surface determined by $\tilde{\Sigma}$. The same argument as above gives a proper minimal immersion from $\Omega \cup \{\gamma_1,\ldots,\gamma_k\}$. Choose a sufficiently small $r$ such that $F^{-1}(B_r(Q) \cap \Sigma)$ is a disjoint union of open sets $D_1,\ldots, D_k$ with $\gamma_i \in D_i$. Since $\Sigma$ is properly embedded, by the maximum principle, we have $F(D_i)=F(D_j)$ for all $i,j$. Therefore, $\Sigma \cup \{Q\}$ is a smooth properly embedded compact minimal surface in $M$ with smooth free boundary on $\partial M$. This proves Theorem 4.1.

\end{proof}


\section{Curvature Estimates}

In this section, we extend the well-known ``small total curvature'' estimate of Choi and Schoen \cite{Choi-Schoen85} to the free boundary case.

\begin{theorem}
Let $M^3$ be a compact Riemannian 3-manifold with boundary. Then there exists $r>0$ small enough (depending only on $M$ and $\partial M$) such that the following holds: let $Q \in M$,  and suppose $\Sigma$ is a compact properly immersed minimal surface in $M$ with free boundary on $\partial M$ such that $Q \in \Sigma$. Then, there exists $\epsilon>0$ depending only on the geometry of $B_r(Q)$ in $M$ such that if 
\begin{equation*}
\int_{\Sigma \cap B_r(Q)} \| h^\Sigma\|^2 da \leq \epsilon,
\end{equation*}
then we have 
\begin{equation*}
\max_{0 \leq \sigma \leq r} \left( \sigma^2 \sup_{B_{r-\sigma}(Q)} \| h^\Sigma \|^2 \right) \leq C,
\end{equation*}
where $C$ is a constant depending only on the geometry of $B_r(Q)$ in $M$.
\end{theorem}

\begin{proof}
Choose $\sigma_0 \in (0,r]$ such that 
\begin{equation*}
\sigma_0^2 \sup_{B_{r-\sigma_0}(Q)} \|h^\Sigma\|^2 = \max_{0 \leq \sigma \leq r} \left( \sigma^2 \sup_{B_{r-\sigma}(Q)} \| h^\Sigma \|^2 \right).
\end{equation*}
Let $Q_0 \in B_{r-\sigma_0}(Q)$ be chosen so that 
\begin{equation*}
\|h^\Sigma\|^2(Q_0)=\sup_{B_{r-\sigma_0}(Q)} \|h^\Sigma\|^2.
\end{equation*}
Therefore, as $B_{\sigma_0/2}(Q_0) \subset B_{r-\sigma_0/2}(Q)$, by the choice of $\sigma_0$ and $Q_0$, we have 
\begin{equation*}
\sup_{B_{\sigma_0/2}(Q_0)} \|h^\Sigma\|^2 \leq \sup_{B_{r-\sigma_0/2}(Q)} \|h^\Sigma\|^2 \leq 4 \|h^\Sigma\|^2(Q_0).
\end{equation*}
We rescale the metric $ds^2$ on $M$ by setting $\tilde{ds}^2=\|h^\Sigma\|^2(Q_0) ds^2$. Then $\Sigma$ is still a minimal surface with respect to $\tilde{ds}^2$ with free boundary on $\partial M$. In the rescaled metric, we have 
\begin{equation}
\|\tilde{h}^\Sigma\|^2(Q_0)=1 \qquad \text{and} \qquad \sup_{\tilde{B}_{r_0}(Q_0)} \|\tilde{h}^\Sigma\|^2 \leq 4,
\end{equation}
where $r_0=\frac{1}{2} \sigma_0 \|h^\Sigma\|(Q_0)$. We claim that we can choose $\epsilon$ sufficiently small enough so that $r_0 \leq 1$. In this case, 
\begin{equation*}
\max_{0 \leq \sigma \leq r} \left( \sigma^2 \sup_{B_{r-\sigma}(Q)} \| h^\Sigma \|^2 \right)=\sigma_0^2 \sup_{B_{r-\sigma_0}(Q)} \|h^\Sigma\|^2  \leq 4.
\end{equation*}
Suppose that $r_0 \geq 1$, then $\rho_0=(\|h^\Sigma\|(Q_0))^{-1} \leq \sigma_0/2$. Therefore, 
\begin{equation*}
\int_{\Sigma \cap \tilde{B}_1(Q_0)} \|\tilde{h}^\Sigma\|^2 \tilde{da} = \int_{\Sigma \cap B_{\rho_0}(Q_0)} \| h^\Sigma \|^2 da \leq \int_{\Sigma \cap B_r(Q)} \| h^\Sigma \|^2 da \leq \epsilon.
\end{equation*}
If $0<\epsilon<1$, then 
\begin{equation}
\int_{\Sigma \cap \tilde{B}_{\epsilon^{1/3}}(Q_0)} \|\tilde{h}^\Sigma\|^2 \tilde{da}  \leq \int_{\Sigma \cap \tilde{B}_1(Q_0)} \|\tilde{h}^\Sigma\|^2 \tilde{da} \leq \epsilon.
\end{equation}
By the monotonicity formula for minimal surfaces with free boundary (\cite{Gruter-Jost86a}), when $r$ is sufficiently small, there is a constant $c>0$, depending on the geometry of $B_r(Q)$, such that the area with respect to $\tilde{ds}^2$ satisfies $A(\Sigma \cap \tilde{B}_{\epsilon^{1/3}}(Q_0)) \geq c \epsilon^{2/3}$. Thus, (5.2) implies that 
\begin{equation*}
\inf_{\Sigma \cap \tilde{B}_{\epsilon^{1/3}}(Q_0)} \| \tilde{h}^\Sigma\|^2 \leq c^{-1} \epsilon^{1/3}.
\end{equation*}
Together with (5.2), this implies that 
\begin{equation}
\| \| \tilde{h}^\Sigma\|^2 \|_{C^{0,\alpha}(\Sigma \cap \tilde{B}_{\epsilon^{1/3}}(Q_0))} \geq c \epsilon^{-\alpha /3}
\end{equation}
for any $\alpha \in (0,1)$. On the other hand, using (5.1), for each $P \in \Sigma \cap \tilde{B}_{\epsilon^{1/3}}(Q_0)$, the connected component of  $\Sigma \cap \tilde{B}_{\epsilon^{1/3}}(Q_0)$ containing $P$ is a graph over some open set of $T_P \Sigma$ of a function $u^P$ with uniformly bounded gradient and Hessian. Note that $\Sigma$ is minimal, hence $u$ satisfies a uniformly elliptic equation. If $P \notin \partial M$, we can apply the interior Schauder estimate (Corollary 6.3 of \cite{Gilbarg-Trudinger01}) to get a uniform $C^{2,\alpha}$ estimate for $u^P$. If $P \in \partial M$, using the free boundary condition, $u$ satisfies a homogeneous boundary condition in  the Fermi coordinates, hence the Schauder estimate for uniformly elliptic equations with homogeneous boundary conditions (\cite{Agmon-Douglis-Nirenberg59}) again implies a uniform $C^{2,\alpha}$ estimate for $u^P$. Therefore, in any case, we have 
\begin{equation*}
\| \| \tilde{h}^\Sigma\|^2 \|_{C^{0,\alpha}(\Sigma \cap \tilde{B}_{\epsilon^{1/3}}(Q_0))} \leq C
\end{equation*}
for some constant $C>0$ depending only on the geometry of $B_r(Q)$. This contradicts (5.3) above when $\epsilon>0$ is sufficiently small (depending on the geometry of $B_r(Q)$). As a result, when $\epsilon>0$ is chosen small enough, then $r_0 \leq 1$. So we are done.

\end{proof}


\section{The Smooth Compactness Theorem}

We prove our main compactness result in this section.

\begin{theorem}
Let $M^3$ be a compact $3$-dimensional Riemannian manifold with nonempty boundary $\partial M$. Suppose $M$ has nonnegative Ricci curvature and the boundary $\partial M$ is strictly convex with respect to the inward unit normal. Then the space of compact properly embedded minimal surfaces of fixed topological type in $M$ with free boundary on $\partial M$ is compact in the $C^k$ topology for any $k \geq 2$.
\end{theorem}

\begin{proof}
Note that by Theorem 2.11, $M^3$ is diffeomorphic to the unit ball $B^3$, hence is simply connected. Let $\Sigma$ be a compact properly embedded minimal surface with free boundary on $\partial M$. Then $\Sigma$ is orientable. Suppose $\Sigma$ has genus $g$ with $\gamma$ boundary components. From the Gauss equation and the minimality of $\Sigma$, for any $x \in \Sigma$, we have
\begin{equation*}
\frac{1}{2} \|h^\Sigma\|^2(x) =K^M(x) - K^\Sigma(x), 
\end{equation*}
where $K^M(x)$ and $K^\Sigma(x)$ are the sectional curvatures of the plane $T_x \Sigma$ with respect to $M$ and $\Sigma$ respectively. We can Integrate the equality above over $\Sigma$ and apply the Gauss-Bonnet theorem to obtain
\begin{equation*}
\frac{1}{2} \int_\Sigma \|h^\Sigma \|^2 = \int_\Sigma K^M + \int_{\partial \Sigma} k_g - 2\pi \chi (\Sigma),
\end{equation*}
where $k_g$ is the geodesic curvature of $\partial \Sigma$ with respect to $\Sigma$ and $\chi(\Sigma)$ is the Euler characteristic of $\Sigma$. Since $\Sigma$ meets $\partial M$ orthogonally along $\partial \Sigma$, $k_g$ is equal to $h^{\partial M}(u,u)$, where $u$ is the unit tangent vector for $\partial \Sigma$. Therefore, there exists a constant $C>0$ depending only on the upper bound of the sectional curvature of $M$ and the principal curvatures of $\partial M$ so that 
\begin{equation*}
\frac{1}{2} \int_\Sigma \|h^\Sigma \|^2  \leq C A(\Sigma) +C L(\partial \Sigma) -2\pi (2-2g-\gamma).
\end{equation*}
Using the isoperimetric inequality (2.2) and the apriori length bound in Proposition 3.4, we obtain
\begin{equation*}
 \int_\Sigma \|h^\Sigma \|^2  \leq C (g +\gamma),
\end{equation*}
where $C$ is a constant depending only on the geometry of the ambient manifold $M$.

Let $\{\Sigma_i\}$ be a sequence of compact properly embedded minimal surfaces of fixed topological type. Using the same covering argument on \cite[P.390-391]{Choi-Schoen85}, we can extract a subsequence of $\{\Sigma_i\}$, which we still call $\{\Sigma_i\}$, and a finite number of points $\{x_1,\ldots,x_\ell\}$ such that $\Sigma_i$ converges in the $C^\infty$ topology to some $\Sigma_0$ in $M \setminus \cup_{j=1}^\ell B_r(x_j)$ for any sufficiently small $r>0$. Here, $\Sigma_0$ is a properly embedded minimal surface (possibly with multiplicity) in $M \setminus \{x_1,\ldots,x_\ell\}$ with free boundary on $\partial M \setminus \{ x_1,\ldots,x_\ell\}$. Note that some $x_j$ may lie on $\partial M$. By the removable singularity theorem (Theorem 4.1), $\Sigma=\Sigma_0 \cup \{ x_1,\ldots, x_\ell\}$ is a compact properly embedded minimal surface with free boundary. The only thing left to prove is that $\Sigma$ has multiplicity $1$ as the limit of $\Sigma_i$. 

Recall that $\Sigma$ is orientable. As $\Sigma_i$ converges to $\Sigma_0$ in $M \setminus \cup_{j=1}^\ell B_{\epsilon^2}(x_j)$ for any sufficiently small $\epsilon$, there exists a large enough $i$ such that $\Sigma_i \setminus \cup_{j=1}^\ell B_{\epsilon^2} (x_j)$ is locally a union of graphs over $\Sigma_0$ by the curvature estimate in Theorem 5.1. We claim that there is only one sheet. Suppose not, then since $\Sigma$ is orientable,we can order the sheets $\Gamma_1,\ldots,\Gamma_k$, where $k \geq 2$. We claim that in this case, we would have $\sigma_1(\Sigma_i) \to 0$ as $\epsilon \to 0$, which would contradict the eigenvalue estimate in Theorem 3.1.

To prove that $\sigma_1(\Sigma_i) \to 0$ as $\epsilon \to 0$, we define a Lipschitz function on $\Sigma_i$ such that 
\begin{equation*}
\varphi= \left\{ \begin{array}{cl}
1 & \text{on } \Gamma_1 \setminus \cup_{j=1}^\ell B_\epsilon(x_j) \\
\frac{\log r_j - \log \epsilon^2}{\log \epsilon - \log \epsilon^2} & \text{on each } \Gamma_1 \cap (B_\epsilon(x_j) \setminus B_{\epsilon^2}(x_j)) \\
0 & \text{on } \Sigma_i \cap \cup_{j=1}^\ell B_{\epsilon^2}(x_j) \\
-\frac{\log r_j - \log \epsilon^2}{\log \epsilon - \log \epsilon^2} & \text{on each } (\Gamma_2 \cup \cdots \cup \Gamma_k) \cap (B_\epsilon(x_j) \setminus B_{\epsilon^2}(x_j)) \\
-1 & \text{on } (\Gamma_2 \cup \cdots \cup \Gamma_k) \setminus \cup_{j=1}^\ell B_\epsilon(x_j), \end{array} \right.
\end{equation*}
where $r_j=d^M(x_j,\cdot)$ is the distance function in $M$ from $x_j$. After possibly subtracting a constant, we can assume that $\int_{\partial \Sigma_i} \varphi=0$. Using the coarea formula and the monotonicity formula for minimal surfaces with free boundary (\cite{Gruter-Jost86a}), the same calculation as on \cite[P.392]{Choi-Schoen85} implies that 
\begin{equation*}
\int_{\Sigma_i} \| \nabla^{\Sigma_i} \varphi \|^2 \to 0 \qquad \text{as } \epsilon \to 0.
\end{equation*}
On the other hand, $\int_{\partial \Sigma_i} \varphi^2$ converges to a constant $C$ as $\epsilon \to 0$. Since each $\Gamma_\ell$ covers $\partial \Sigma_0$ once, we have $L(\partial \Gamma_\ell) \geq L(\partial \Sigma_0)-\eta$ for any arbitrarily small $\eta>0$ as $\epsilon \to 0$. Hence, $C\neq 0$. By the variational characterization of the first Steklov eigenvalue
\begin{equation*}
\sigma_1(\Sigma_i)=\inf_{\int_{\partial \Sigma_i} f =0, f \notequiv 0} \frac{\int_{\Sigma_i} \| \nabla^{\Sigma_i} f \|^2}{\int_{\partial \Sigma_i} f^2},
\end{equation*}
we see that $\sigma_1(\Sigma_i) \to 0$ as $\epsilon \to 0$. Using the Allard regularity theorem for minimal surfaces with free boundary (\cite{Gruter-Jost86}), we see that $\Sigma_i$ converges to $\Sigma$ in the $C^\infty$ topology even across the points $x_1,\ldots,x_\ell$. This completes the proof of Theorem 6.1.
\end{proof}

\begin{corollary}
The space of compact properly embedded smooth minimal surfaces of fixed topological type in the Euclidean unit ball $B^3$ with free boundary on $\partial B^3$ is compact in the $C^\infty$ topology.
\end{corollary}

\bibliographystyle{amsplain}
\bibliography{references}

\end{document}